\documentclass[11pt]{article}
\usepackage{amsfonts}
\usepackage{color}
\usepackage{amsmath,amssymb,comment}
\newtheorem{theo}{Theorem}[section]
\newtheorem{lem}[theo]{Lemma}

\newtheorem{prop}[theo]{Proposition}

\newtheorem{defi}[theo]{Definition}

\newcommand{\mysection}[1]{\section{#1} \setcounter{equation}{0}}
\newcommand{\proof}{{\sc Proof.} \quad}

\newcommand{\be}{\begin{equation} \label}
\newcommand{\ee}{\end{equation}}
\newcommand{\bea}{\begin{eqnarray}\label}
\newcommand{\eea}{\end{eqnarray}}
\newcommand{\bas}{\begin{eqnarray*}}
\newcommand{\eas}{\end{eqnarray*}}
\newcommand{\bit}{\begin{itemize}}
\newcommand{\eit}{\end{itemize}}
\newcommand{\qed}{\hfill$\Box$ \vskip.2cm}
\newcommand{\nn}{\nonumber}

\newcommand{\eps}{\varepsilon}

\newcommand{\io}{\int_\Omega}
\newcommand{\na}{\nabla}
\newcommand{\Del}{\Delta}

\newcommand{\Om}{\Omega}

\newcommand{\ov}{\overline}

\newcommand{\lbal}{\left\{ \begin{array}{l}}
\newcommand{\lball}{\left\{ \begin{array}{ll}}
\newcommand{\ear}{\end{array} \right.}

\newcommand{\abs}{\\[5pt]}

\newcommand{\adb}{\allowdisplaybreaks}

\newcommand{\uT}{\underline{\Theta}}
\newcommand{\uTeps}{{\underline{\Theta}}_\eps}

\newcommand{\bfi}{\textnormal{\textbf{i}}}
\newcommand{\bfj}{\textnormal{\textbf{j}}}
\newcommand{\bfk}{\textnormal{\textbf{k}}}
\newcommand{\bfl}{\textnormal{\textbf{l}}}
\begin{document}
\adb
%
%
\title{Quantitative positivity properties of the temperature in multi-dimensional thermoviscoelastic systems}
\author{
Torben Fricke\footnote{tjfricke@math.uni-paderborn.de}\\
{\small Universit\"at Paderborn, Institut f\"ur Mathematik}\\
{\small 33098 Paderborn, Germany} }
\date{}
\maketitle
\begin{abstract}
\noindent We consider the classical model for small-strain thermoviscoelastic evolution in isotropic Kelvin-Voigt materials with inertia, as given by
\bas
	\lbal
	u_{tt}=\mu\Del u_t+(\lambda+\mu)\na\na\cdot u_t+\hat{\mu}\Del u+(\hat{\lambda}+\hat{\mu})\na\na\cdot u-B\na\Theta,\\
	\kappa\Theta_t=D\Del\Theta+\mu|\na u_t|^2+(\lambda+\mu)|\text{div }u_t|^2-B\Theta\text{div }u_t.
	\ear
\eas
Recently, \cite{Win26} established global-in-time solvability for an initial value problem in smoothly bounded multi-dimensional domains $\Om$ under homogeneous Dirichlet boundary conditions for the displacement variable $u$, and homogeneous Neumann boundary conditions for the temperature $\Theta$, allowing the heat capacity $\kappa$ to depend on $\Theta$. Moreover, \cite{Win26} proves that the corresponding temperature is positive a.e. in $\Om\times (0,\infty)$.\abs
In the present work, we strengthen this positivity result by deriving a quantitative lower bound for $\Theta$. Our analysis applies to heat capacities $\kappa$ with power-law scaling for small arguments, precisely: either $\kappa(\xi)\sim\xi^\alpha$ for some $\alpha\ge 0$, or $\kappa(\xi)\ge\xi$, in a suitable neighborhood of $\xi=0$. Under these structural assumptions on $\kappa$, we show that for every $\tau>0$ and $T>\tau$ there exists $c(\tau,T)>0$ such that
$$\Theta(x,t)\ge c(\tau,T)\qquad\text{for a.e. }(x,t)\in\Om\times(\tau,T).$$
\end{abstract}
\newpage
\section{Introduction}\label{intro}
Thermoviscoelastic models capture the interplay between mechanical deformations and temperature evolution in viscoelastic materials. Over the past decades, these systems have been studied extensively, with pioneering works (\cite{Daf1}, \cite{Daf2}) establishing foundational techniques. Subsequent research has addressed global solvability (\cite{Jiang}, \cite{Kim}, \cite{Racke}, \cite{Watson}) and long-time behavior (\cite{Hsiao}, \cite{Racke}) in one-dimensional settings, while higher-dimensional problems remain more challenging. \abs
For the classical model of nonlinear thermoviscoelastic dynamics at small strain in isotropic Kelvin-Voigt materials, for instance as given by (\cite{Pawlow13}, \cite{Roubi09}, \cite{Roubi10}), the system 
\be{0}
	\lball
	u_{tt}=\mu \Del u_t+(\lambda+\mu)\na\na\cdot u_t+\hat{\mu}\Del u+(\hat{\lambda}+\hat{\mu})\na\na\cdot u-B\na \Theta,\\[1mm]
	\kappa\Theta_t=D\Del\Theta+\mu |\na u_t|^2+(\lambda+\mu)|\text{div }u_t|^2-B\Theta\text{div }u_t
	\ear
\ee
with constant parameters $\mu,\lambda,\hat{\mu},\hat{\lambda},B,\kappa$ and $D$ has been studied extensively. Various modifications have been considered, including temperature-dependent heat capacities at high temperatures, viscous dissipation at large strains, and saturation effects in thermal expansion (\cite{Blanchard}, \cite{Gawinecki16}, \cite{Pawlow13}, \cite{Pawlow17}, \cite{Roubi09}); see also \cite{Bartels}, \cite{Guo}, \cite{Owczarek}, \cite{Paoli1}, \cite{Paoli2}, \cite{Pawlow00}, \cite{Rossi}, \cite{Roubi10}, \cite{Roubi13} and \cite{Roubi13b} for related problems involving fourth-order terms, plasticity, or adhesive contact.\abs 
For the system \eqref{0}, local-in-time solvability is established in \cite{Bonetti}, whereas \cite{Win26} proves global existence for multi-dimensional domains and discusses large-time stabilization of the temperature. In contrast, previous results for small-data smooth solutions (\cite{Gawinecki12}, \cite{Shibata}) or under strong symmetry and parameter restrictions (\cite{Cieslak}, \cite{Gawinecki03}) do not cover the general large-data case.\abs
Following the framework of \cite{Win26}, we consider a more general formulation with temperature-dependent heat capacity and possibly nonlinear stress-strain relationships. Specifically, we study
\be{sys0}
	\lball
	u_{tt}=\text{div }(\mathbb{D}:\na^s u_t)+\text{div }(\mathbb{C}:\na^s u)-\text{div }(\Theta\mathbb{B})+f(x,t),\qquad& x\in\Om,\ t>0,\\[1mm]
	\kappa(\Theta)\Theta_t=D\Del\Theta+\langle\mathbb{D}:\na^s u_t,\na^s u_t\rangle-\Theta\langle\mathbb{B},\na^s u_t\rangle+g(x,t),\qquad& x\in\Om,\ t>0,\\[1mm]
	u=0,\quad \frac{\partial\Theta}{\partial\nu}=0,& x\in\partial\Om,\ t>0,\\[1mm]
	u(x,0)=u_0(x),\quad u_t(x,0)=u_{0t}(x),\quad \Theta(x,0)=\Theta_0(x), & x\in\Om,
	\ear
\ee
posed in smoothly bounded domain $\Om\subset\mathbb{R}^n$. Here $D>0$, the tensors $\mathbb{D},\ \mathbb{C}\in\mathbb{R}^{n\times n\times n\times n}$ represent the viscous and elastic moduli, $\mathbb{B}\in\mathbb{R}^{n\times n}$ couples mechanics and temperature, $f$ and $g$ are prescribed forcing terms, $\kappa:[0,\infty)\to[0,\infty)$ is a possibly temperature-dependent heat capacity, and $u_0$, $u_{0t}$, $\Theta_0$ are given initial data. This is the setting in which \cite{Win26} established global existence and positivity of solutions; we refine those results by proving a quantitative lower bound for $\Theta$.\abs
\textbf{Main results.}\quad 
As noted above, the results in \cite{Win26} are already quite strong, establishing global existence, stabilization of the temperature, and positivity: $\Theta>0$ a.e. in $\Om\times(0,\infty)$. However, for many applications in thermoviscoelastic systems, a uniform lower bound of the form $\Theta\geq c(t)>0$ is often desired (\cite{Win-Rough} for instance). In this work, we complement the analysis of \cite{Win26} by deriving such a quantitative lower bound.\abs
More precisely, within the framework of the generalized concept of solvability introduced in \cite{Win26} (cf. Definition \ref{DefWeak} below), we first recall the main properties of solutions obtained there:
\begin{prop}\label{theo1}
Let $n\ge 2$ and $\Om\subset\mathbb{R}^n$ be a bounded domain with smooth boundary, $D>0$, and assume that $\mathbb{D}=(\mathbb{D}_{\bfi\bfj\bfk\bfl})_{(\bfi,\bfj,\bfk,\bfl)\in\{1,\dots,n\}^4},\ \mathbb{C}=(\mathbb{C}_{\bfi\bfj\bfk\bfl})_{(\bfi,\bfj,\bfk,\bfl)\in\{1,\dots,n\}^4}\in\mathbb{R}^{n\times n\times n\times n}$ and $\mathbb{B}=(\mathbb{B}_{\bfi\bfj})_{(\bfi,\bfj)\in\{1,\dots,n\}^2}\in\mathbb{R}_{sym}^{n\times n}$ satisfying $\mathbb{D}_{\bfi\bfj\bfl\bfk}=\mathbb{D}_{\bfk\bfl\bfi\bfj}=\mathbb{D}_{\bfj\bfi\bfk\bfl}$ and $\mathbb{C}_{\bfi\bfj\bfl\bfk}=\mathbb{C}_{\bfk\bfl\bfi\bfj}=\mathbb{C}_{\bfj\bfi\bfk\bfl}$ for all $(\bfi,\bfj,\bfk,\bfl)\in\{1,\dots,n\}^4$, with constants $k_\mathbb{D}>0$ and $k_\mathbb{C}>0$ such that
\bea{t1.3}
\langle\mathbb{D}:\mathbb{A},\mathbb{A}\rangle\ge k_\mathbb{D}|\mathbb{A}|^2\quad\text{and}\quad\langle\mathbb{C}:\mathbb{A},\mathbb{A}\rangle\ge k_\mathbb{C}|\mathbb{A}|^2\qquad\text{for all }\mathbb{A}\in\mathbb{R}_{sym}^{n\times n}.
\eea
Suppose that $f\in L_{loc}^1([0,\infty);L^2(\Om;\mathbb{R}^n))$, $g\in L_{loc}^1(\ov\Om\times[0,\infty))$ with $g\ge 0$, and $\kappa\in C^0([0,\infty))$ with $\kappa>0$ on $(0,\infty)$. Let the initial data satisfy
\be{t1.7}
	\lball
	u_0\in W_0^{1,2}(\Om;\mathbb{R}^n),\\[1mm]
	u_{0t}\in L^2(\Om,\mathbb{R}^n)\qquad\text{and}\\[1mm]
	\Theta_0:\Om\to\mathbb{R}\quad\text{is such that}\quad\Theta_0\not\equiv 0,\quad K(\Theta_0)\in L^1(\Om)\quad\text{and}\quad l(\Theta_0)\in L^1(\Om),
	\ear
\ee
where
\bea{t1.8}
K(\xi):=\int_0^\xi \kappa(\sigma)d\sigma,\qquad \xi\ge 0,\quad\text{and}\quad l(\xi):=\int_1^\xi\frac{\kappa(\sigma)}{\sigma}d\sigma,\qquad \xi>0.
\eea
(A1)\quad If $\kappa$ satisfies $\liminf_{\xi\to\infty}\kappa(\xi)>0$, then there exist
\be{t1.10}
	\lball
	u\in C^0([0,\infty);L^2(\Om,\mathbb{R}^n))\cap L_{loc}^\infty([0,\infty);W_0^{1,2}(\Om;\mathbb{R}^n))\qquad\text{and}\\
	\Theta\in L_{loc}^\infty([0,\infty);L^1(\Om))
	\ear
\ee
such that $\Theta>0$ a.e. in $\Om\times(0,\infty)$, and that $(u,\Theta)$ forms a generalized solution of \eqref{sys0} in the sense of Definition \eqref{DefWeak} i).\abs
(A2)\quad If $\kappa(\xi)\cdot\ln^{(3-n)_+}\xi\to+\infty$ as $\xi\to\infty$,
then there exists a global integrable solution $(u,\Theta)$ of \eqref{sys0} in the sense of Definition \ref{DefWeak} ii) with $\Theta>0$ a.e. in $\Om\times(0,\infty)$. Moreover, if $\liminf_{\xi\to\infty}\kappa(\xi)>0$, then $\Theta\in L_{loc}^\infty([0,\infty);L^1(\Om))$.
\end{prop}
Building upon this, we then establish our main result, which provides a quantitative lower bound for the temperature.
\begin{theo}\label{theo2}
Assume, that in addition to Proposition \ref{theo1}, there exists $\eta>0$ such that $\kappa$ is nondecreasing on $(0,\eta)$ and that $\kappa$ satisfies one of the following two structural conditions:\abs
(B1)\quad There exist constants $\alpha\ge 0$ and $c_\kappa>0$ such that
\bea{A1}
\frac{1}{c_\kappa}\xi^\alpha\ge \kappa(\xi)\ge c_\kappa \xi^\alpha\qquad\text{for all }\xi\in(0,\eta);
\eea
(B2)\quad There exists a constant $c_\kappa>0$ such that
\bea{A2}
\kappa(\xi)\ge c_\kappa \xi\qquad\text{for all }\xi\in(0,\eta).
\eea
Then for each $\tau>0$ and $T>\tau$, there exists $c(\tau,T)>0$ such that
\bea{A3}
\Theta(x,t)\ge c(\tau,T)\qquad\text{for a.e. }(x,t)\in \Om\times(\tau,T).
\eea
\end{theo}
\textbf{Remark.}\quad Note that condition (B1) with $\alpha \in [0,1]$ implies condition (B2). Consequently, in the case $\alpha \in [0,1]$, condition (B2) already applies, 
and the special analysis for (B1) is needed only when $\alpha > 1$. 
This reduction is used implicitly in the proof of Lemma \ref{Lem5.3}.\abs
\textbf{Sketch of proof.}\quad 
The main idea is to exploit mass conservation of the temperature. Since $\Theta_0\not\equiv 0$, the total thermal energy $\int_\Omega K(\Theta)$ remains positive for $t>0$. We construct a subsolution $\underline{\Theta}$ to the temperature equation satisfying
\bea{t1.20}
\kappa(\underline{\Theta})\underline{\Theta}_t\ge a\Del\underline{\Theta}-C\underline{\Theta}^2
\eea
for suitable $a>0$, which by comparison ensures $\underline{\Theta}\le\Theta$. Mass preservation implies that $\Theta$ is bounded away from zero on a subset of positive measure in any small time interval $[t_0, t_0 + \tau)$. To convert this subdomain positivity into a uniform bound, we use the logarithmic substitution $z = -\ln \Theta$, which satisfies a parabolic equation of the form
$$
z_t = \frac{D \Delta z}{\kappa(\Theta)} - \frac{D|\nabla z|^2}{\kappa(\Theta)} + \frac{a\Theta}{\kappa(\Theta)}.
$$
Using the localization to the positive-measure region and a Poincaré-type inequality, we first establish $L^1$-bounds on $z$. A Moser iteration then yields $L^\infty$-bounds on $z$ over $(t_0+\frac{\tau}{2}, t_0+\tau)$, which translate to quantitative lower bounds $\Theta \ge c(\tau) > 0$. Finally, since mass is conserved at each time level, we iterate this construction forward in time to extend the lower bound over any interval $(\tau, T)$.
\mysection{Preliminaries}\label{sec:2}
\subsection{Generalized and integrable solutions}
We follow the solution concept introduced in \cite{Win26}. Let $\mathcal{M}(\ov\Omega)$ denote the Banach space of signed Radon measures on $\ov\Omega$ and $L^\infty_{w-*}((0,\infty);\mathcal{M}(\ov\Omega))$ the space of weak-$*$ measurable functions $\mu:(0,\infty)\to\mathcal{M}(\ov\Omega)$, with $\|\mu\|_{L^\infty_{w-*}} := \operatorname{ess\,sup}_{t>0}\|\mu(t)\|_{\mathcal{M}(\ov\Omega)}$.
\begin{defi}\label{DefWeak}
Let $\Om\subset\mathbb{R}^n$ be a bounded domain with smooth boundary, $D>0$, and let $\mathbb{D},\ \mathbb{C},\ \mathbb{B}$ be as in Proposition \ref{theo1}. Suppose $f\in L_{loc}^1([0,\infty);L^2(\Om,\mathbb{R}^n))$, $g\in L_{loc}^1(\ov\Om\times[0,\infty))$ with $g\ge 0$, $\kappa\in C^0([0,\infty))$, and initial data $u_0\in W_0^{1,2}(\Om;\mathbb{R}^n)$, $u_{0t}\in L^2(\Om,\mathbb{R}^n)$ and $\Theta_0\in L^1(\Om)$ with $\Theta_0\ge 0$ a.e. in $\Om$ and $K(\Theta_0)\in L^1(\Om)$, where $K(\xi):=\int_0^\xi \kappa(\sigma)d\sigma$, $\xi\ge 0$.\abs
i)\quad By a generalized solution of \eqref{sys0} we mean a pair of functions
\be{2.1}
	\lball
	u\in L_{loc}^\infty([0,\infty);W_0^{1,2}(\Om;\mathbb{R}^n))\qquad\text{and}\\[1mm]
	\Theta\in L_{loc}^1(\ov\Om\times[0,\infty))
	\ear
\ee
such that $u_t\in L_{loc}^\infty([0,\infty);L^2(\Om;\mathbb{R}^n))$ and $\na^s u_t\in L_{loc}^1(\ov\Om\times[0,\infty);\mathbb{R}^{n\times n})$, moreover $\Theta\ge 0$ a.e. in $\Om\times(0,\infty)$ and $K(\Theta)\in L_{loc}^\infty([0,\infty);L^1(\Om))$ and $\na\phi_0(\Theta)\in L_{loc}^2(\ov\Om\times[0,\infty);\mathbb{R}^n)$ for all $\phi_0\in C^1([0,\infty))$ with compact support. Furthermore, there exist $\mu\in L_{w^{-\star},loc}^\infty([0,\infty);\mathcal{M}_+(\ov\Om))$ and $a>0$ such that
\bea{2.6}
    \int_0^\infty\int_\Omega u&&\hspace{-9mm}\cdot\ \varphi_{tt} + \int_\Omega u_0\cdot\varphi_t(\cdot,0) 
    - \int_\Omega u_{0t}\cdot\varphi(\cdot,0) \nn\\
    &=& -\int_0^\infty\int_\Omega\langle\mathbb{D}:\nabla^s u_t,\nabla\varphi\rangle 
    - \int_0^\infty\int_\Omega\langle\mathbb{C}:\nabla^s u,\nabla\varphi\rangle \notag\\
    && + \int_0^\infty\int_\Omega\Theta\langle\mathbb{B},\nabla\varphi\rangle 
    + \int_0^\infty\int_{\ov\Omega}\langle\mathbb{B},\nabla\varphi\rangle d\mu 
    + \int_0^\infty\int_\Omega f\cdot\varphi
\eea
holds for all $\varphi\in C_0^\infty(\Omega\times[0,\infty);\mathbb{R}^n)$, as well as
\bea{2.7}
    \mathcal{F}(t) + a\int_{\ov\Omega}d\mu(t) \leq \mathcal{F}_0 
    + \int_0^t\int_\Omega f\cdot u_t + \int_0^t\int_\Omega g 
    \qquad \text{for a.e. } t>0,
\eea
with
\bea{2.8}
    \mathcal{F}(t) &:=& \frac{1}{2}\int_\Omega|u_t(\cdot,t)|^2 
    + \frac{1}{2}\int_\Omega\langle\mathbb{C}:\nabla^s u(\cdot,t),\nabla^s u(\cdot,t)\rangle 
    + \int_\Omega K(\Theta(\cdot,t)), \qquad t>0,\label{2.8}\\
    \mathcal{F}_0 &:=& \frac{1}{2}\int_\Omega|u_{0t}|^2 
    + \frac{1}{2}\int_\Omega\langle\mathbb{C}:\nabla^s u_0,\nabla^s u_0\rangle 
    + \int_\Omega K(\Theta_0),\label{2.9}
\eea
and such that for each nondecreasing concave $\phi\in C^\infty([0,\infty))$ with 
$\phi'\in C_0^\infty([0,\infty))$,
\bea{2.10}
    -\int_0^\infty\int_\Omega&&\hspace{-9mm} K^{(\phi)}(\Theta)\varphi_t 
    - \int_\Omega K^{(\phi)}(\Theta_0)\varphi(\cdot,0)\nn\\ 
    &\ge& -D\int_0^\infty\int_\Omega\phi''(\Theta)|\nabla\Theta|^2\varphi 
    - D\int_0^\infty\int_\Omega\phi'(\Theta)\nabla\Theta\cdot\nabla\varphi+ \int_0^\infty\int_\Omega\phi'(\Theta)\langle\mathbb{D}:\nabla^s u_t,\nabla^s u_t\rangle\varphi \notag\\
    && - \int_0^\infty\int_\Omega\Theta\phi'(\Theta)\langle\mathbb{B},\nabla^s u_t\rangle\varphi 
    + \int_0^\infty\int_\Omega\phi'(\Theta)g\varphi
\eea
for all nonnegative $\varphi\in C_0^\infty(\ov\Omega\times[0,\infty))$ holds, where $K^{(\phi)}(\xi) := \int_0^\xi\kappa(\sigma)\phi'(\sigma)\,d\sigma.$\abs
ii)\quad A generalized solution $(u,\Theta)$ of \eqref{sys0} is called an \emph{integrable solution} of \eqref{sys0} if in the above we may choose $\mu\equiv 0$.
\end{defi}
\subsection{A regularized problem and known results}\label{sec:2.2}
We recall the regularized problem from \cite{Win26} and collect necessary results.\abs
Specifically, here and below assume that $n\ge 2$ and that $\Om\subset\mathbb{R}^n$ bounded with smooth boundary, and let $D$, $\mathbb{D}$, $\mathbb{C}$, $\mathbb{B}$, $f$, $g$ and $\kappa$ be as in Proposition \ref{theo1}. Let the functions $K$ and $l$ be as defined in \eqref{t1.8}, and suppose that the initial data $u_0$, $u_{0t}$ and $\Theta_0$ are such that \eqref{t1.7} holds. Fix $\eta>0$ such that $\kappa$ is nondecreasing on $(0,\eta)$. By standard smoothing procedures, there exist families $(\kappa_\eps)_{\eps\in(0,1)}\subset C^\infty([0,\infty))$ with
\bea{kappa-bed}
\kappa_\eps(\xi)\ge\eps,\ |\kappa_\eps(\xi)-\kappa(\xi)|\le\frac{\eps}{(\xi+1)^2}\text{ for all }\xi\ge 0 \text{ and }\kappa_\eps'(\xi)\ge 0\text{ for all }\xi\in(0,\eta)\text{ and }\eps\in(0,1)\nn\\
\eea
and
\be{2.12}
	\lball
	(f_\eps)_{\eps\in(0,1)}\subset C_0^\infty(\Om\times(0,\infty);\mathbb{R}^n)\qquad\text{and}\\[1mm]
	(g_\eps)_{\eps\in(0,1)}\subset C_0^\infty(\Om\times(0,\infty))\text{ with }g_\eps\ge 0\text{ for all }\eps\in(0,1),
	\ear
\ee
as well as
\be{2.13}
	\lball
	(v_{0\eps})_{\eps\in(0,1)}\subset C_0^\infty(\Om;\mathbb{R}^n),\\[1mm]
	(u_{0\eps})_{\eps\in(0,1)}\subset C_0^\infty(\Om;\mathbb{R}^n)\qquad\text{and}\\[1mm]
	(\Theta_{0\eps})_{\eps\in(0,1)}\subset C^\infty(\ov\Om)\quad\text{with }\na\Theta_{0\eps}\in C_0^\infty(\Om;\mathbb{R}^n)\text{ and }\Theta_{0\eps}\ge \eps\text{ for all }\eps\in(0,1),
	\ear
\ee
such that as $\eps\searrow 0$ we have
\bea{2.14}
f_\eps\to f\text{ in }L_{loc}^1([0,\infty);L^2(\Om;\mathbb{R}^n))\quad\text{ and }g_\eps\to g\text{ in }L_{loc}^1(\ov\Om\times[0,\infty))
\eea
and
\bea{2.15}
v_{0\eps}\to u_{0t}\text{ in }L^2(\Om),\quad u_{0\eps}\to u_0\text{ in }W^{1,2}(\Om)\quad\text{and}\quad \Theta_{0\eps}\to\Theta_0\text{ a.e.}
\eea
as well as
\bea{2.16}
K(\Theta_{0\eps})\to K(\Theta_0)\quad\text{and}\quad l(\Theta_{0\eps})\to l(\Theta_0)\qquad\text{in }L^1(\Om),
\eea
and that for each $\eps_0\in(0,1)$, the sets $(\kappa_\eps)_{\eps>\eps_0}$, $(f_\eps)_{\eps>\eps_0}$, $(g_\eps)_{\eps>\eps_0}$, $(v_{0\eps})_{\eps>\eps_0}$, $(u_{0\eps})_{\eps>\eps_0}$ and $(\Theta_{0\eps})_{\eps>\eps_0}$ are all finite. As shown in \cite[Lemma 3.3]{Win26}, for $m\ge 1$ with $m>\frac{n+2}{4}$ and $\eps\in(0,1)$, the parabolic regularization of \eqref{sys0} is given by
\be{sys1}
	\lball
	v_{\eps t}+\eps \Del^{2m}v_\eps=\text{div }(\mathbb{D}:\na^s v_\eps)+\text{div }(\mathbb{C}:\na^s u_\eps)-\text{div }(\Theta_\eps\mathbb{B})+f_\eps(x,t),\qquad& x\in\Om,\ t>0,\\[1mm]
	u_{\eps t}=v_\eps,\qquad& x\in\Om,\ t>0,\\[1mm]
	\kappa_\eps(\Theta_\eps)\Theta_{\eps t}=D\Del\Theta_\eps+\langle\mathbb{D}:\na^s v_\eps,\na^s v_\eps\rangle-\Theta_\eps\langle\mathbb{B},\na^s v_\eps\rangle+g_\eps(x,t),\qquad& x\in\Om,\ t>0,\\[1mm]
	v_\eps=0,\quad u_\eps=0,\quad \frac{\partial\Theta_\eps}{\partial\nu}=0,& x\in\partial\Om,\ t>0,\\[1mm]
	v_\eps(x,0)=v_{0\eps}(x),\quad u_\eps(x,0)=u_{0\eps}(x),\quad \Theta_\eps(x,0)=\Theta_{0\eps}(x), & x\in\Om
	\ear
\ee
and admits local classical solutions in the following sense. 
\begin{lem}\label{lem2.2}
Let $m\in\mathbb{N}$ satisfy $m>\frac{n+2}{4}$. For each $\eps\in(0,1)$ there exist
\be{2.18}
	\lball
	v_\eps\in C^0([0,\infty);W_0^{2m,2}(\Om,\mathbb{R}^n))\cap C^\infty(\ov\Om\times(0,\infty);\mathbb{R}^n),\\[1mm]
	u_\eps\in C^1([0,\infty);W_0^{2m,2}(\Om;\mathbb{R}^n))\cap C^\infty(\ov\Om\times(0,\infty);\mathbb{R}^n)\qquad\text{and}\\[1mm]
	\Theta_\eps\in C^{2,1}(\ov\Om\times[0,\infty))\cap C^\infty(\ov\Om\times(0,\infty))
	\ear
\ee
such that $\Theta_\eps>0$ in $\ov\Om\times[0,\infty)$ and that \eqref{sys1} is solved in the classical sense in $\Om\times(0,\infty)$.
\end{lem}
The analysis carried out in \cite[Section 8]{Win26} yields the following convergence results for these solutions.
\begin{lem}
If $\liminf_{\xi\to\infty}\kappa(\xi)>0$, then there exist $(\eps_j)_{j\in\mathbb{N}}\subset(0,1)$ and a generalized solution $(u,\Theta)$ of \eqref{sys0}, in the sense of Definition \ref{DefWeak} i) such that $\eps_j\searrow 0$ as $j\to\infty$ and moreover
\begin{alignat}{2}
&v_\eps \rightharpoonup u_t \qquad &&\text{in }L^2(\Om\times(0,T); \mathbb{R}^n),\label{2.19}\\
&\na^s v_\eps \rightharpoonup \na^s u_t \qquad &&\text{in }L^1(\Om\times(0,T);\mathbb{R}^n),\label{2.20}\\
&u_\eps\to u\qquad &&\text{in }C^0([0,T];L^2(\Om;\mathbb{R}^n)),\label{2.21}\\
&u_\eps\rightharpoonup u\qquad &&\text{in }L^2((0,T);W^{1,2}(\Om;\mathbb{R}^n)),\label{2.22}\\
&\Theta_\eps\to\Theta\qquad &&\text{a.e. in }\Om\times(0,T),\label{2.23}\\
&\Theta_\eps \overset{*}{\rightharpoonup}\Theta+\mu\qquad &&\text{in }L^\infty((0,T);\mathcal{M}_+(\ov\Om))\qquad\text{and}\label{2.24}\\
&\na\phi_0(\Theta_\eps)\rightharpoonup\na\phi_0(\Theta)\qquad &&\text{in }L^2(\Om\times(0,T);\mathbb{R}^n)\label{2.25}
\end{alignat}
as $\eps=\eps_j\searrow0$.
\end{lem}
\begin{lem}
If $\kappa(\xi)\cdot\ln^{(3-n)_+}\xi\to\infty$ as $\xi\to\infty$, one can find $(\eps_j)_{j\in\mathbb{N}}\subset(0,1)$ as well as an integrable solution $(u,\Theta)$ of \eqref{sys0} in the sense of Definition \ref{DefWeak} ii) satisfying $\eps_j\searrow 0$ as $j\to\infty$ and \eqref{2.19}, \eqref{2.20}, \eqref{2.21}, \eqref{2.22}, \eqref{2.23}, \eqref{2.25} and
\bea{nn}
\Theta_\eps\to\Theta\qquad\text{in }L^1(\Om\times(0,T))\nn
\eea
as $\eps=\eps_j\searrow0$.
\end{lem}
The approximate solutions established in \cite{Win26} satisfy the following key properties (see \cite{Win26}, Lemma 3.2 and Lemma 9.2):
\begin{lem}\label{W2}
Assume $\kappa\in C^0([0,\infty))$ with $\kappa>0$ on $(0,\infty)$, and for $\eps\in(0,1)$ let
\bea{K_eps_def}
K_\eps(\xi):=\int_0^\xi \kappa_\eps(\sigma)d\sigma,\qquad\xi\ge 0.
\eea
Then
\bea{K_eps_abs}
|K_\eps(\xi)-K(\xi)|\le\eps\qquad\text{for all }\xi\ge 0.
\eea
\end{lem}
\begin{lem}\label{W1}
There exist $b>0$, $\eps_0\in(0,1)$ and $(C(T))_{T>0}\subset(0,\infty)$ such that whenever $T>0$,
\bea{nn}
\big|\{\Theta_\eps(\cdot,t)>b\}\big|\ge\frac{1}{C(T)}\qquad\text{for all }t\in(0,T)\text{ and }\eps\in(0,\eps_0).\nn
\eea
\end{lem}
\mysection{An approximate subsolution and mass estimates}
In this section we construct for each $\eps\in(0,1)$ a subsolution $\uTeps$ to the approximate system \eqref{sys1} and derive key mass and first positivity estimates. The main result is that a certain notion of mass persists for short times, and that positivity of $\uTeps$ can be established on a subdomain of positive measure, uniformly in $\eps$.
\subsection{Construction of the subsolution and basic properties}
Fix $T>0$ and $\delta>0$ (to be restricted to be sufficiently small at various points below, by introducing thresholds $\delta_0$ and $\delta_1$; see Lemma \ref{lem3.5} and Lemma \ref{lem4.6}). For arbitrary $t_0\in[0,T]$ and $\eps\in(0,1)$ we construct a subsolution to \eqref{sys1} as follows. Set $a:=\frac{|\mathbb{B}|^2}{4k_\mathbb{D}}$, where $\mathbb{B}$ and $\mathbb{D}$ are as in Proposition \ref{theo1} and $k_\mathbb{D}>0$ is the coercivity constant in \eqref{t1.3}. Consider
\be{sys3}
	\lball
	\kappa_\eps(\uT_\eps)\uT_{\eps t}=D\Del \uT_\eps-a\uT_\eps^2,\qquad& x\in\Om,\ t>t_0,\\[1mm]
	\frac{\partial\uT_\eps}{\partial\nu}=0,\qquad& x\in\partial\Om,\ t>t_0,\\[1mm]
	\uT_\eps(x,t_0):=\min\{\Theta_{\eps}(x,t_0),e^{-1},\delta\},\qquad& x\in\Om.
	\ear
\ee
We note that $\uTeps$ depends on both $t_0$ and $\delta$; for simplicity of notation, this dependence is suppressed. The following lemma ensures well-posedness of \eqref{sys3} and establishes that $\uTeps$ indeed is a subsolution to $\Theta_\varepsilon$ in \eqref{sys1}.
\begin{lem}
Let $t_0\in[0,T]$ and $\eps\in(0,1)$. Then \eqref{sys3} admits a unique classical solution
$$\underline{\Theta}_\varepsilon \in C^{2,1}(\overline{\Omega} \times [t_0,\infty)) \cap C^\infty(\overline{\Omega} \times (t_0,\infty))$$
with $\uTeps>0$ in $\ov\Om\times[t_0,\infty)$. Moreover, $\uTeps\le \Theta_\eps$ in $\ov\Om\times[t_0,\infty)$.
\end{lem}
\proof
Existence, uniqueness and positivity of $\uTeps$ follow from standard parabolic theory, noting that $\kappa_\eps>0$ allows one to rewrite \eqref{sys3} in non-divergence form. For the comparison, observe that $\uTeps(\cdot,t_0)\le \Theta_{\eps}(\cdot,t_0)$ in $\Om$ by definition. Moreover, using \eqref{2.12} and $k_\mathbb{D}>0$ as in \eqref{t1.3}, Young's inequality yields
\bea{e1}
\kappa_\eps(\uT_\eps)\uT_{\eps t}&=&D\Del\uT_\eps+\langle\mathbb{D}:\na^s v_\eps,\na^s v_\eps\rangle-\uT_\eps\langle\mathbb{B},\na^s v_\eps\rangle+g_\eps(x,t)\nn\\
&\ge&D\Del\uT_\eps+k_\mathbb{D}|\na^s v_\eps|^2-|\mathbb{B}|\uTeps|\na^s v_\eps|\nn\\
&\ge&D\Del\uT_\eps+k_\mathbb{D}|\na^s v_\eps|^2-k_\mathbb{D}|\na^s v_\eps|^2-\frac{|\mathbb{B}|^2}{4k_\mathbb{D}}\uT_\eps^2\nn\\
&=&D\Del\uT_\eps-a\uT_\eps^2\qquad\text{in }\Om\times(t_0,\infty),\nn
\eea
for all $\eps\in(0,1)$. Thus a simple comparison argument completes the proof.
\qed
As a direct consequence of the structure of \eqref{sys3}, the subsolution $\uTeps$ inherits the upper bound present in the initial data.
\begin{lem}\label{lemThetaBounded}
Let $t_0\in[0,T]$ and $\eps\in(0,1)$, then 
$$\uTeps\leq \min\{e^{-1},\delta\}\qquad\text{in }\Omega\times(t_0,\infty).$$
\end{lem}
\proof
The constant function $\omega\equiv \min\{e^{-1},\delta\}$ serves as a supersolution for \eqref{sys3}, since $\kappa_\eps(\omega)\omega_t-D\Del \omega+a\omega^2=a\ge 0$ for all $\eps\in(0,1)$.
\qed
\subsection{Positivity on a subdomain}
We first show that for short time scales, the mass of $\uTeps$ does not disappear completly. This will allow us to establish positivity of $\uTeps$ for all $\eps\in(0,1)$ on a suitable subdomain.
\begin{lem}
For $\eps\in(0,1)$, let $K_\eps$ as in \eqref{K_eps_def}. Then there exists $\tau_0>0$ such that for all $t_0\in[0,T]$
\bea{p3.1}
\io K_\eps(\uTeps(\cdot,t))\ge \frac{1}{2}\inf_{\eps'\in(0,1)}\io K_{\eps'}(\underline{\Theta}_{\eps'}(\cdot,t_0))\quad\text{for all }t\in(t_0,t_0+\tau_0),
\eea
where, by Lemma \ref{W1} and \eqref{kappa-bed}, the infimum on the right-hand side is positive.
\end{lem}
\proof
From the definition of $K_\eps$ and the subsolution equation \eqref{sys3}, we have for all $\varepsilon \in (0,1)$,
\bea{e.p3.2}
\frac{d}{dt}\io K_\eps(\uTeps)=\io \kappa_\eps(\uTeps)\uT_{\eps t}= -a\io \uTeps^2\qquad\text{for all }t>t_0.\nn
\eea
Integrating over time and using Lemma \ref{lemThetaBounded} yield
\bea{e.p3.3}
\io K_\eps(\uTeps(\cdot,t_0+t))&=&\io K_\eps(\uTeps(\cdot,t_0)-a\int_{t_0}^{t_0+t}\io \uTeps^2(\cdot,s)ds\nn\\
&\ge&\io K_\eps(\uTeps(\cdot,t_0)- a t|\Om|e^{-2}\qquad\text{for all }t\in(0,1)\text{ and }\eps\in(0,1).\nn
\eea
Let $c_2:=\inf_{\eps'\in(0,1)}\io K_{\eps'}(\underline{\Theta}_{\eps'}(\cdot,t_0)$. Setting $\tau_0:=\frac{c_2 e^2}{2a|\Om|},$ we obtain
$$\io K_\eps(\uTeps(\cdot,t))\ge\frac{c_2}{2}\qquad\text{for all }t\in(t_0,t_0+\tau_0)\text{ and }\eps\in(0,1),$$
which yields \eqref{p3.1}.
\qed
We now establish that the mass persistence from the above lemma implies that $\uTeps$ is bounded away from zero on a set of positive measure.
\begin{lem}\label{lem3.4}
Let $\tau_0>0$ be as in the lemma above. Then there exist constants $C_1>0$, $C_2>0$ and $\eps_1\in(0,1)$ such that for each $\eps\in(0,\eps_1)$, any $t_0\in[0,T]$ and arbitrary $t\in(t_0,t_0+\tau_0)$, we can find a measurable set $E_\eps(t)\subset\Om$ with
\bea{p4.1-1}
\int_{E_{\eps}(t)}K_\eps(\uTeps(\cdot,t))\geq\frac{1}{4}\inf_{\eps'\in(0,1)}\io K_{\eps'}(\underline{\Theta}_{\eps'}(\cdot,t_0)\qquad\text{for all }\eps\in(0,\eps_1),
\eea
$$|E_\eps(t)|\ge C_1\qquad\text{for all }\eps\in(0,\eps_1),$$
and
\bea{p4.1}
\uTeps(x,t)\ge C_2\qquad \text{for all }x\in E_\eps(t)\text{ and }\eps\in(0,\eps_1).
\eea
\end{lem}
\proof
By \eqref{p3.1} we have
\bea{e.p4.1}
\io K_\eps(\uTeps(\cdot,t))\ge c_1:=\frac{1}{2}\inf_{\eps'\in(0,1)}\io K_{\eps'}(\underline{\Theta}_{\eps'}(\cdot,t_0)\qquad\text{for all }t\in(t_0,t_0+\tau_0)\text{ and }\eps\in(0,1).
\eea
Since \eqref{K_eps_abs} holds and $K(0)=0$, we can choose $\eps_1\in(0,1)$ small enough so that there exists $c_2>0$ with $|\Om|\cdot K_\eps(c_2)<\frac{c_1}{2}$ for all $\eps\in(0,\eps_1)$. For fixed $t\in(t_0,t_0+\tau_0)$, define
$$E_\eps(t):=\big\{x\in\Om\ \big|\ \uTeps(x,t)\ge c_2\big\}\qquad\text{for all }\eps\in(0,\eps_1).$$
Splitting the integral yields
$$c_1\le \io K_\eps(\uTeps(\cdot,t))\le \int_{E_\eps(t)}K_\eps(\uTeps(\cdot,t))+|\Om|\cdot K_\eps(c_2)\le\int_{E_\eps(t)}K_\eps(\uTeps(\cdot,t))+\frac{c_1}{2},$$
for all $t\in(t_0,t_0+\tau_0)$ and $\eps\in(0,\eps_1)$. Thus, $\int_{E_\eps(t)} K_\eps(\uTeps(\cdot,t))\ge\frac{c_1}{2}$ for all $t\in(t_0,t_0+\tau_0)$ and $\eps\in(0,\eps_1)$. Together with Lemma \ref{lemThetaBounded} we finish the proof with some $C_1>0$ small enough and $C_2:=c_2$. 
\qed
In the case where \eqref{A2} holds, the lower bound \eqref{p4.1} will be sufficient for our purposes. If instead \eqref{A1} holds, we will later need to choose $\delta>0$ arbitrarily small. To this end, we first record the following technical properties of $\kappa_\eps$.
\begin{lem}
Assume that \eqref{A1} holds with some $\alpha \geq 0$ on $(0,\eta)$, for some $\eta>0$. Then there exists a constant $C_1 > 0$ independent of $\delta$ such that
\bea{kappa-eps-bed}
C_1\Big\{\xi+\eps^\frac{1}{\alpha}\Big\}^\alpha\ge\kappa_\eps(\xi)\ge \frac{1}{C_1}\Big\{\xi+\eps^\frac{1}{\alpha}\Big\}^\alpha\qquad\text{for all }\xi\in(0,\eta)\text{ and }\eps\in(0,1).
\eea
If \eqref{A2} holds in place of \eqref{A1}, then there exists $C_2>0$ such that
\bea{kappa-eps-bed2}
\kappa_\eps(\xi)\ge \frac{1}{C_2}(\xi+\eps)\qquad\text{for all }\xi\in(0,\eta)\text{ and }\eps\in(0,1).
\eea
\end{lem}
\proof
From \eqref{kappa-bed} we have $|\kappa_\eps(\xi)-\kappa(\xi)|\le \eps$ for all $\xi\ge 0$ and $\eps\in(0,1)$. For the upper bound, recalling $c_\kappa\in(0,1]$ from \eqref{A1}, we obtain
\bea{k-b-d-1}
\kappa_\eps(\xi)\le \kappa(\xi)+\eps\le \frac{1}{c_\kappa}\xi^\alpha+\eps\le \frac{1}{c_\kappa}\Big\{\xi+\eps^{\frac{1}{\alpha}}\Big\}^\alpha\qquad\text{for all }\xi\in(0,\eta)\text{ and }\eps\in(0,1).
\eea
For the lower bound, \eqref{kappa-bed} and \eqref{A1} give $\kappa_\eps(\xi)\ge \max\{\kappa(\xi)-\eps,\ \eps\}\ge \max\{c_\kappa \xi^\alpha-\eps,\ \eps\}$ for all $\xi\in(0,\eta)$ and $\eps\in(0,1)$, whence
\bea{k-b-d-2}
3\kappa_\eps(\xi)\ge c_\kappa \xi^\alpha-\eps +2\eps=c_\kappa\xi^\alpha+\eps\ge \frac{c_\kappa}{2^{\alpha-1}}\Big\{\xi+\eps^\frac{1}{\alpha}\Big\}^\alpha\qquad\text{for all }\xi\in(0,\eta)\text{ and }\eps\in(0,1).
\eea
Together, \eqref{k-b-d-1} and \eqref{k-b-d-2} establish \eqref{kappa-eps-bed} with $C_1$ large enough.\abs For \eqref{kappa-eps-bed2}, only the lower bound is needed, and the same argument as in \eqref{k-b-d-2} applies with $\alpha=1$: \eqref{kappa-bed} and \eqref{A2} give $\kappa_\eps\ge \max\{c_\kappa\xi-\eps,\ \eps\}$ for all $\xi\in(0,\eta)$ and $\eps\in(0,1)$, whence, recalling that $c_\kappa\in(0,1]$ as before,
$$3\kappa_\eps(\xi)\ge c_\kappa\xi+\eps\ge c_\kappa(\xi+\eps)\qquad\text{for all }\xi\in(0,\eta)\text{ and }\eps\in(0,1).$$
\qed
With these in hand, we can now  establish that $|E_\eps(t)|$ remains bounded away from zero independently of $\delta$:  
\begin{lem}\label{lem3.5}
In the situation of Lemma \ref{lem3.4} (with constants $\tau_0,\ E_\eps$ understood to depend on $\delta$), assume that \eqref{A1} holds with some $\alpha\ge 0$ on $(0,\eta)$, for some $\eta>0$. Then there exist $\delta_0\in(0,\eta)$ and $\eps_2\in(0,\eps_1)$ such that for all $\delta\in(0,\delta_0)$, one can find $C>0$, independent of $\delta$ and $\eps$, such that $|E_\eps(t)|\ge C$ for all $t_0\in[0,T]$, $t\in(t_0,t_0+\tau_0)$ and $\eps\in(0,\eps_2)$. 
\end{lem}
\proof
By \eqref{kappa-eps-bed}, we fix $c_1>0$ independent of $\delta$ such that 
$$\frac{1}{c_1} \Big\{\xi+\eps^\frac{1}{\alpha}\Big\}^\alpha\leq \kappa_\eps(\xi)\le c_1\Big\{ \xi+\eps^\frac{1}{\alpha}\Big\}^\alpha\qquad\text{for all }\xi\in(0,\eta)\text{ and }\eps\in(0,1).$$ 
Due to Lemma \ref{W1} and \eqref{2.15}, one can find $c_2>0$, $\delta_0\in(0,\eta)$ and $\eps_2\in(0,\eps_1)$ such that $\uTeps(\cdot,t_0)\ge\delta$ on a set $\hat{E}_\eps\subset\Om$ with $|\hat{E}_\eps|\ge c_2>0$ for all $\eps\in(0,\eps_2)$, where $c_2$ is independent of $\eps$ and $\delta$. So
$$\inf_{\eps'\in(0,\eps_2)}\io K_{\eps'}(\underline{\Theta}_{\eps'}(\cdot,t_0))\ge\inf_{\eps'\in(0,\eps_2)}\int_{\hat{E}_{\eps'}}K_{\eps'}(\delta)\ge \frac{c_2}{c_1(\alpha+1)}\delta^{\alpha+1}.$$

For $\eps\in[\eps_2,1)$, by Lemma \ref{lem2.2}, $\Theta_\eps>0$ on $\ov\Om\times[0,\infty)$, and since this range of $\eps$ is finite, we may shrink $\delta_0$ so that $\uTeps(\cdot,t_0)\ge \delta$ in $\Om$ for all $\eps\in[\eps_2,1)$. Thus we can find $c_3>0$ independent of $\delta$ such that $\inf_{\eps'\in(0,1)}\io K_{\eps'}(\underline{\Theta}_{\eps'}(\cdot,t_0)\ge c_3\delta^{\alpha+1}$. From Lemma \ref{lemThetaBounded} we infer $$K_\eps(\uTeps)\le K_\eps(\delta)\le \frac{c_1 2^{\alpha+1}}{\alpha+1}\delta^{\alpha+1}\qquad\text{in }\Om\times(t_0,\infty),$$
for all $\eps\in(0,1)$ and thus with \eqref{p4.1-1}
$$\frac{c_3\delta^{\alpha+1}}{4}\le \int_{E_\eps(t)}K_\eps(\uTeps(\cdot,t))\le \frac{c_12^{\alpha+1}\delta^{\alpha+1}}{\alpha+1}|E_\eps(t)|\qquad\text{for all }t\in(t_0,t_0+\tau_0)\text{ and }\eps\in(0,\eps_2).$$
Since $c_1$ and $c_3$ are independent of $\delta$, we obtain $|E_\eps(t)|\ge C:=\frac{c_3}{2^{\alpha+3}c_1}(\alpha+1)$ for all $t\in(t_0,t_0+\tau_0)$ and $\eps\in(0,\eps_2)$.
\qed
\mysection{Quantitative lower bounds for $\Theta$}\label{sec:4}
The primary goal of this section is to establish pointwise lower bounds for $\uTeps$, in the spirit of \cite[Section 6]{Win-Rough}. Our analysis to this end is initiated by the following observation.
\begin{lem}
For $\eps\in(0,1)$ and $t_0\in[0,T]$, let
$$z_\eps(x,t):=-\ln (\uTeps(x,t)),\qquad x\in\overline{\Om},\ t\ge t_0,$$
then $z_\eps$ lies in $C^{2,1}(\ov\Om\times[t_0,\infty))$ and satisfies
\be{sys5}
	\lball
	z_{\eps t}=\frac{D\Del z_\eps}{\kappa_\eps(\uTeps)}-\frac{D|\na z_\eps|^2}{\kappa_\eps(\uTeps)}+\frac{a\uTeps}{\kappa_\eps(\uTeps)},\qquad& x\in\Om,\ t>t_0,\\[1mm]
	\frac{\partial z_\eps}{\partial\nu}=0,\qquad& x\in\partial\Om,\ t>t_0,\\[1mm]
	z_\eps(x,t_0)=-\ln(\underline{\Theta}_{\eps}(x,t_0)),\qquad& x\in\Om,
	\ear
\ee
as well as
\bea{g5.2}
z_\eps(x,t)\ge 1\qquad\text{in }\Om\times(t_0,\infty).
\eea
\end{lem}
\proof
The inequality \eqref{g5.2} follows directly from Lemma \ref{lemThetaBounded} and the definition of $z_\eps$. Since $z_{\eps t}=-\frac{\underline{\Theta}_{\eps t}}{\uTeps}$ and $\Del z_\eps=-\frac{\Del \uTeps}{\uTeps}+\frac{|\na\uTeps|^2}{\uTeps^2}$ for all $\eps\in(0,1)$, the remaining terms can be verified by straightforward computation on the basis of \eqref{sys3}.
\qed
\subsection{Auxiliary estimates}
We rely on the following two auxiliary results: a localized Poincar\'e inequality and an elementary ODE comparison argument.
\begin{lem}\label{AE1}
Let $\Om\subset\mathbb{R}^n$ be a bounded domain, and let $\iota>0$. Then there exists $C=C(\Om,\iota)>0$ such that whenever $E\subset\Om$ is measurable with $|E|\ge\iota$,
$$\|\varphi\|_{L^2(\Om)}\le C\|\na\varphi\|_{L^2(\Om)}+C\bigg|\int_E\varphi\bigg|\qquad\text{for all }\varphi\in W^{1,2}(\Om).$$
\end{lem}
\proof
Assume for contradiction that the assertion were false. Then one can find measurable sets $(E_k)_{k\in\mathbb{N}}\subset\Om$ with $|E_k|\ge\iota$ for all $k\in\mathbb{N}$, and $(\varphi_k)_{k\in\mathbb{N}}\subset W^{1,2}(\Om)$ satisfying
$$\int_{E_k}\varphi_k\to 0\quad\text{and}\quad \|\na\varphi_k\|_{L^2(\Om)}\to 0\qquad \text{as }k\to\infty.$$
Without loss of generality, we may assume that $\|\varphi_k\|_{L^2(\Om)}=1$ for all $k\in\mathbb{N}$. Since $(\varphi_k)$ is bounded in $W^{1,2}(\Om)$, the compact embedding $W^{1,2}(\Om)\hookrightarrow L^2(\Om)$ yields (after extracting a subsequence) $\varphi_k\to\varphi$ strongly in $L^2(\Om)$. Furthermore, $\na\varphi=0$ in $L^2(\Om)$, so $\varphi\equiv c$ for some constant $c\in\mathbb{R}$ with $c^2|\Om|=1$, in particular $c\neq 0$. By the Cauchy-Schwarz inequality,
$$\bigg|\int_{E_k}(\varphi_k-c)\bigg|\le \sqrt{|E_k|}\cdot\|\varphi_k-c\|_{L^2(\Om)}\to0\qquad\text{as }k\to\infty,$$ so that $$c|E_k|=\int_{E_k}\varphi_k-\int_{E_k}(\varphi_k-c)\to0\qquad\text{as }k\to\infty.$$ Since $|E_k|\ge\iota$ for all $k\in\mathbb{N}$, this forces $c=0$, thereby contradicting $c\neq 0$. 
\qed
\begin{lem}\label{AE2}
Let $p>0$, $q\ge 0$, $t_0\in\mathbb{R}$ and $T>t_0$, and suppose that $y\in C^0([t_0,T))\cap C^1((t_0,T))$ is such that
\bea{nn}
y'(t)+py^2(t)\leq q\qquad\text{for all }t\in(t_0,T).\nn
\eea
Then
\bea{nn}
y(t)\le \frac{1}{p(t-t_0)}+\sqrt{\frac{q}{p}}\qquad\text{for all }t\in(t_0,T).\nn
\eea
\end{lem}
\proof
Since for $\ov y(t):=\frac{1}{p(t-t_0)}+\sqrt{\frac{q}{p}}$, $t>t_0$, we have $\ov y(t)\to+\infty$ as $t\searrow t_0$ and
\bea{nn}
\ov y'(t)+p\ov y^2(t)-q&=&-\frac{1}{p(t-t_0)^2}+p\cdot\Big\{\frac{1}{p(t-t_0)}+\sqrt{\frac{q}{p}}\Big\}^2-q\nn\\
&\ge&-\frac{1}{p(t-t_0)^2}+p\cdot\Big\{\Big(\frac{1}{p(t-t_0)}\Big)^2+\Big(\sqrt{\frac{q}{p}}\Big)^2\Big\}-q=0\nn
\eea
for all $t>t_0$. Using that continuity implies local boundedness of $y$ in $[t_0,T)$ we may invoke a comparison argument to derive the result.
\qed
\subsection{$L^1$ and $L^p$ bounds for $z_\eps$}\label{sec5.2}
With the auxiliary estimates at hand, we first establish that $z_\eps(\cdot,t)$ is in $L^1(\Om)$ for fixed $t>t_0$.
\begin{lem}
Let $t_0\in[0,T]$ and let $\eta>0$ be such that $\kappa_\eps'(\xi)\ge0$ for all $\xi\in(0,\eta)$ and $\eps\in(0,1)$. If $\delta\in(0,\eta)$ and $p\ge 1$ then
\bea{g5.3}
\frac{1}{p}\frac{d}{dt}\io z_\eps^p+D\io |\na z_\eps|^2 z_\eps^{p-2} \frac{z_\eps+p-1}{\kappa_\eps(\uTeps)}\le a\io \frac{z_\eps^{p-1}\uTeps}{\kappa_\eps(\uTeps)}\qquad\text{for all }t>t_0\text{ and }\eps\in(0,1).
\eea
\end{lem}
\proof
Let $p\ge 2$. From \eqref{sys5} we immediately obtain
\bea{g4.4}
\frac{1}{p}\frac{d}{dt}\io z_\eps^p&=&D\io \frac{z_\eps^{p-1}\Del z_\eps}{\kappa_\eps(\uTeps)}-D \io\frac{z_\eps^{p-1}|\na z_\eps|^2}{\kappa_\eps(\uTeps)}+a\io \frac{z_\eps^{p-1}\uTeps}{\kappa_\eps(\uTeps)}\quad\text{for all }t>t_0\text{ and }\eps\in(0,1).\ \ \ \
\eea
Integrating by parts and using that $\na z_\eps=-\frac{\na \uTeps}{\uTeps}$, equivalently $\na \uTeps=-\uTeps\na z_\eps$, and 
$$\na\bigg(\frac{1}{\kappa_\eps(\uTeps)}\bigg)=-\frac{\kappa_\eps'(\uTeps)}{\kappa_\eps(\uTeps)^2}\na \uTeps,$$
valid in $\Om\times(t_0,\infty)$ and for all $\eps\in(0,1)$, we obtain that
\bea{nn}
D\io \frac{z_\eps^{p-1}\Del z_\eps}{\kappa_\eps(\uTeps)}&=&-(p-1)D\io \frac{z_\eps^{p-2}|\na z_\eps|^2}{\kappa_\eps(\uTeps)}-D\io z_\eps^{p-1}\na\Big(\frac{1}{\kappa_\eps(\uTeps)}\Big)\cdot\na z_\eps\nn\\
&=&-(p-1)D\io \frac{z_\eps^{p-2}|\na z_\eps|^2}{\kappa_\eps(\uTeps)}-D\io z_\eps^{p-1}\frac{\kappa_\eps'(\uTeps)\uTeps}{\kappa_\eps(\uTeps)^2}|\na z_\eps|^2\nn
\eea
for all $t>t_0$ and $\eps\in(0,1)$, which together with \eqref{g4.4} yields
\bea{nn}
\frac{1}{p}\frac{d}{dt}\io z_\eps^p=-D\io |\na z_\eps|^2 z_\eps^{p-2}\bigg\{\frac{z_\eps+p-1}{\kappa_\eps(\uTeps)}+\frac{\kappa_\eps'(\uTeps)\uTeps z_\eps}{\kappa_\eps(\uTeps)^2}\bigg\}+a\io\frac{z_\eps^{p-1}\uTeps}{\kappa_\eps(\uTeps)}\nn
\eea
for all $t>t_0$ and $\eps\in(0,1)$. By Lemma \ref{lemThetaBounded}, $\uTeps\leq\delta<\eta$ in $\Om\times(t_0,\infty)$ for all $\eps\in(0,1)$, so that, by hypothesis, $\kappa_\eps'(\uTeps)\ge 0$ in $\Om\times(t_0,\infty)$ for all $\eps\in(0,1)$.  Therefore, $$\io |\na z_\eps|^2z_\eps^{p-1}\frac{\kappa_\eps'(\uTeps)\uTeps}{\kappa_\eps(\uTeps)^2}\ge 0,$$ and \eqref{g5.3} follows. The case $p=1$ can be checked separately by the same argument, which does not require any essential modifications.
\qed
\begin{lem}\label{lem4.6}
Let $\eta>0$ be such that $\kappa_\eps'(\xi)\ge0$ for all $\xi\in(0,\eta)$ and $\eps\in(0,1)$, and assume that \eqref{A1} holds for some $\alpha\ge 0$, or \eqref{A2} holds, both on $(0,\eta)$. Let $\delta_0\in(0,\eta)$ and $\eps_2\in(0,1)$ be as in Lemma \ref{lem3.5}. Then there exists $\delta_1\in(0,\delta_0)$ such that for all $\delta\in(0,\delta_1)$, with $\tau_0>0$ as in  Lemma \ref{lem3.4}  we can find $C>0$ fulfilling
\bea{l4.6}
\io z_\eps(\cdot,t)\le \frac{C}{t-t_0}\qquad\text{for all }t_0\in[0,T],\ t\in(t_0,t_0+\min\{1,\tau_0\})\text{ and }\eps\in(0,\eps_2).
\eea
\end{lem}
\proof
Due to Lemma \ref{lem3.4} there exists $c_1>0$ and $c_2>0$ such that for all $t\in(t_0,t_0+\tau_0)$ and $\eps\in(0,\eps_2)$ we can find $E_\eps(t)\subset\Om$ with $|E_\eps(t)|\ge c_1$, such that
\bea{4.7}
\uTeps(x,t)\ge c_2\qquad\text{for all }x\in E_\eps(t).
\eea
Since $|E_\eps(t)|\ge c_1$ for all $t\in(t_0,t_0+\tau_0)$ and $\eps\in(0,\eps_2)$, Lemma \ref{AE1} (applied with $\iota:=c_1$) provides a single constant $c_3>0$, valid uniformly over all these sets $E_\eps(t)$, such that
\bea{4.8}
\|\varphi\|_{L^2(\Om)}^2\le c_3\|\na\varphi\|_{L^2(\Om)}^2+c_3\bigg|\int_{E_\eps(t)}\varphi\bigg|^2\quad\text{for all }\varphi\in W^{1,2}(\Om),\ t\in(t_0,t_0+\tau_0)\text{ and }\eps\in(0,\eps_2).
\eea
Case 1: Assume that \eqref{A1} holds for some $\alpha>0$. By Lemma \ref{lem3.5}, the constant $c_1$ may be chosen independently of $\delta$, and consequently so may $c_3$. According to \eqref{kappa-eps-bed} we can fix $c_4>0$ independent of $\delta$, such that 
$$\frac{1}{c_4} \big\{\xi+\eps^\frac{1}{\alpha}\big\}^\alpha\leq \kappa_\eps(\xi)\le c_4\big\{ \xi+\eps^\frac{1}{\alpha}\big\}^\alpha\qquad\text{for all }\xi\in(0,\eta)\text{ and }\eps\in(0,1).$$ 
Now \eqref{g5.3} and Lemma \ref{lemThetaBounded} directly yield
\bea{4.9}
\frac{d}{dt}\io z_\eps&\le&-D\io \frac{|\na z_\eps|^2}{\kappa_\eps(\uTeps)}+a\io \frac{\uTeps}{\kappa_\eps(\uTeps)} \nn\\
&\le& -\frac{D}{c_4}\io |\na z_\eps|^2\big\{\uTeps+\eps^\frac{1}{\alpha}\big\}^{-\alpha}+ac_4\io \uTeps\big\{\uTeps+\eps^\frac{1}{\alpha}\big\}^{-\alpha}\nn\\
&\le& -\frac{D}{c_4}\io |\na z_\eps|^2\big\{\uTeps+\eps^\frac{1}{\alpha}\big\}^{-\alpha}+a\delta c_4\io \big\{\uTeps+\eps^\frac{1}{\alpha}\big\}^{-\alpha}\quad\text{for all }t>t_0\text{ and }\eps\in(0,1). \hspace{10mm}
\eea
Since
\bea{nn}
\Big|\na\big\{e^{-z_\eps}+\eps^\frac{1}{\alpha}\big\}^{-\frac{\alpha}{2}}\Big|=\frac{\alpha}{2}e^{-z_\eps}\big\{e^{-z_\eps}+\eps^\frac{1}{\alpha}\big\}^{-\frac{\alpha}{2}-1}|\na z_\eps|\le \frac{\alpha}{2}\big\{e^{-z_\eps}+\eps^\frac{1}{\alpha}\big\}^{-\frac{\alpha}{2}}|\na z_\eps|\nn
\eea
as well as $2\ge \uTeps+\eps^\frac{1}{\alpha}$ and thus $1\le 2^\alpha\big\{\uTeps+\eps^{\frac{1}{\alpha}}\big\}^{-\alpha}$ for all $t>t_0$ and $\eps\in(0,1)$, by \eqref{4.7}, \eqref{4.8} and Hölder's inequality, we obtain
\bea{4.10}
a\delta c_4\io&&\hspace{-9mm}\big\{\uTeps+\eps^\frac{1}{\alpha}\big\}^{-\alpha}+\delta\bigg\{\io z_\eps\bigg\}^2\nn\\
&\le&a\delta c_4\Big\| \big\{e^{- z_\eps}+\eps^\frac{1}{\alpha}\big\}^{-\frac{\alpha}{2}}\Big\|_{L^2(\Om)}^2+\delta|\Om|\|z_\eps\|_{L^2(\Om)}^2\nn\\
&\le& a\delta c_3c_4\Big\| \na \big\{e^{- z_\eps}+\eps^\frac{1}{\alpha}\big\}^{-\frac{\alpha}{2}}\Big\|_{L^2(\Om)}^2 +a\delta c_3c_4\bigg|\int_{E_\eps(t)}\big\{e^{- z_\eps}+\eps^\frac{1}{\alpha}\big\}^{-\frac{\alpha}{2}}\bigg|^2\nn\\
&&+\delta |\Om|c_3\|\na z_\eps\|_{L^2(\Om)}^2+\delta |\Om|c_3\bigg|\int_{E_\eps(t)}z_\eps\bigg|^2\nn\\
&\le& \frac{a\delta\alpha^2 c_3c_4}{4}\io \big\{e^{-z_\eps}+\eps^\frac{1}{\alpha}\big\}^{-\alpha}|\na z_\eps|^2+a\delta c_3 c_4|E_{\eps}(t)|^2\big\{c_2+\eps^\frac{1}{\alpha}\big\}^{-\alpha}\nn\\
&&+ 2^\alpha\delta|\Om|c_3\io|\na z_\eps|^2\big\{\uTeps+\eps^\frac{1}{\alpha}\big\}^{-\alpha}+\delta|\Om|c_3|E_\eps(t)|^2|\ln(c_2)|^2 \nn\\
&\le&\delta c_5\io |\na z_\eps|^2\big\{\uTeps+\eps^\frac{1}{\alpha}\big\}^{-\alpha}+ c_6\qquad\text{for all }t\in(t_0,t_0+\tau_0)\text{ and }\eps\in(0,\eps_2),
\eea
with $c_5:=\frac{a \alpha^2 c_3c_4}{4}+2^{\alpha}|\Om|c_3$ and $c_6:=a\delta c_2^{-\alpha}c_3c_4|\Om|^2+\delta\ln(c_2)^2c_3|\Om|^3$, since $z_\eps(x,t):=-\ln(\uTeps(x,t))\ge 1$ in $\Om\times(t_0,\infty)$. Since $c_4$ and $c_5$ are independent of $\delta$, we choose $\delta_1:=\min\{\delta_0,\frac{D}{c_4c_5}\}$ to ensure that $\delta c_5\le\frac{D}{c_4}$, so that by \eqref{4.9} and \eqref{4.10}, we infer that $y_\eps(t):=\io z_\eps(\cdot,t)$, $t\in(t_0,t_0+\tau_0)$, $\eps\in(0,\eps_2)$, satisfies
\bea{ret}
y_\eps'(t)+\delta y_\eps^2(t)\le c_6\qquad\text{for all }t\in(t_0,t_0+\tau_0)\text{ and }\eps\in(0,\eps_2).
\eea
An application of Lemma \ref{AE2} shows that, in particular,
$$y_\eps(t)\le \frac{1}{\delta (t-t_0)}+\sqrt{\frac{c_6}{\delta}}\qquad\text{for all }t_0\in(t_0,t_0+\tau_0)\text{ and }\eps\in(0,\eps_2),$$
and thus \eqref{l4.6}.\abs
Case 2: Assume that \eqref{A2} holds. Then by \eqref{kappa-eps-bed} there exists $c_7>0$ such that $\kappa_\eps(\xi)\ge c_7(\xi+\eps)$ for all $\xi\in(0,\eta)$ and $\eps\in(0,1)$. Due to Lemma \ref{lemThetaBounded} and \eqref{kappa-bed}, we can find $c_8>0$ fulfilling $\kappa_\eps(\uTeps)\le c_8$ in $\Om\times(t_0,\infty)$ for all $\eps\in(0,1)$. Now \eqref{g5.3} directly yields
\bea{nn}
\frac{d}{dt}\io z_\eps\le -\frac{D}{c_8}\io |\na z_\eps|^2+\frac{a}{c_7}|\Om| \qquad\text{for all }t>t_0\text{ and }\eps\in(0,1).\nn
\eea
Using \eqref{4.8} and defining $\hat{y}_\eps(t):=\io z_\eps(\cdot,t)$ for $t\in(t_0,t_0+\tau_0)$, $\eps\in(0,\eps_2)$, the same procedure as above (starting with \eqref{ret}) concludes the proof.
\qed
To obtain uniform-in-time estimates for $z_\eps$, we establish appropriate bounds for higher $L^p-$norms. In particular, we derive a differential inequality for $\io z_\eps^p$, which controls the evolution through a dissipative term and a gain term.
\begin{lem}\label{Lem5.3}
Let $\eta>0$ be such that $\kappa_\eps'(\xi)\ge0$ for all $\xi\in(0,\eta)$ and $\eps\in(0,1)$ and fix $\delta\in(0,\eta)$. Assume that either \eqref{A1} for some $\alpha\ge 0$, or \eqref{A2} holds, both on $(0,\eta)$. Let $\tau_0>0$ be as in Lemma \ref{lem3.4}.  Then there exist $\eps_3\in(0,1)$, $C_1>0$, $C_2>0$ and $C_3>0$, such that for all $p\ge 2$, $t_0\in[0,T]$, $t\in(t_0,t_0+\tau_0)$ and $\eps\in(0,\eps_3)$ we have 
\bea{4.14}
\frac{d}{dt}\io z_\eps^p\le -C_1\bigg\{\io z_\eps^p\bigg\}^\frac{n+2}{n}\cdot\bigg\{\io z_\eps^\frac{p}{2}\bigg\}^{-\frac{4}{n}}+C_2p^{C_3}\bigg\{\io z_\eps^\frac{p}{2}\bigg\}^2.
\eea
\end{lem}
\proof
By Lemma \ref{lem3.4}, we can find $\eps_\star\in(0,1)$, $c_1>0$ and $c_2>0$ such that for all $t\in(t_0,t_0+\tau_0)$ and $\eps\in(0,\eps_\star)$ there exists $E_\eps(t)\subset\Om$ with $|E_\eps(t)|\ge c_1$ such that
\bea{4.15}
\uTeps(x,t)\ge c_2\qquad\text{for all }x\in E_\eps(t).
\eea
By Lemma \ref{AE1}, applied with $\iota:=c_1$, we further obtain $c_3>0$, uniform in $t$ and $\eps$, such that
\bea{4.16}
\|\varphi\|_{L^2(\Om)}^2\le c_3\|\na\varphi\|_{L^2(\Om)}^2+c_3\bigg|\int_{E_\eps(t)}\varphi\bigg|^2\qquad\text{for all }\varphi\in W^{1,2}(\Om)\text{ and }\eps\in(0,\eps_\star).
\eea
By the Gagliardo-Nirenberg inequality, there exists $c_4>0$ such that
\bea{GN-1}
\|\varphi\|_{L^2(\Om)}^2\le c_4\|\na\varphi\|_{L^2(\Om)}^\frac{2n}{n+2}\|\varphi\|_{L^1(\Om)}^\frac{4}{n+2}+c_4\|\varphi\|_{L^1(\Om)}^2\qquad\text{for all }\varphi\in W^{1,2}(\Om).
\eea
Case 1: Assume that \eqref{A1} holds for some $\alpha>1$. For $\alpha\le 1$, the argument in Case 2 below applies directly. According to \eqref{kappa-eps-bed} we fix $c_5>0$ such that
\bea{g5.5}
\frac{1}{c_5}\big\{\xi+\eps^\frac{1}{\alpha}\big\}^\alpha\le\kappa_\eps(\xi)\le c_5\big\{\xi+\eps^\frac{1}{\alpha}\big\}^\alpha\qquad\text{for all }\xi\in(0,\eta)\text{ and }\eps\in(0,1).
\eea
Let $q:=\frac{\alpha}{\alpha-1}>1$. Then there exists $\eps_{\star\star}\in(0,\eps_\star)$ such that for all $\eps\in(0,\eps_{\star\star})$, $\uTeps+\eps^\frac{1}{\alpha}\le 1$. Let $c_6:=\frac{4D}{ac_3c_5^2(1+\alpha)^2}$ and $c_7:=\frac{1}{\alpha}q^{1-\alpha}c_6^{1-\alpha}$. Recalling \eqref{g5.3} and applying \eqref{g5.5} and Young's inequality,
\bea{g5.6}
&&\hspace{-20mm}\frac{1}{p}\frac{d}{dt}\io z_\eps^p+D\io|\na z_\eps|^2z_\eps^{p-2}\frac{z_\eps+p-1}{\kappa_\eps(\uTeps)}\nn\\
&\le& a\io \frac{z_\eps^{p-1}\uTeps}{\kappa_\eps(\uTeps)}\nn\\
&\le& a c_5\io z_\eps^{p-1}\cdot\big\{\uTeps+\eps^\frac{1}{\alpha}\big\}^{-\alpha+1}\nn\\
&\le& a c_5\io z_\eps^{p-1}\cdot\big\{\uTeps+\eps^\frac{1}{\alpha}\big\}^{-\frac{\alpha}{q}}\nn\\
&=&ac_5\io z_\eps^{p-1}\cdot\Big\{p^{-2}\big\{\uTeps+\eps^\frac{1}{\alpha}\big\}^{-\alpha}\Big\}^{\frac{1}{q}}\cdot\Big\{p^{2\alpha-2}\Big\}^{\frac{1}{\alpha}}\nn\\
&\le&\frac{ac_5c_6}{p^2}\io z_\eps^{p-1}\cdot\big\{\uTeps+\eps^\frac{1}{\alpha}\big\}^{-\alpha}+ac_5 c_7 p^{2\alpha-2}\io z_\eps^{p-1}
\eea
for all $t>t_0$ and $\eps\in(0,\eps_{\star\star})$. Direct computation shows
\bea{g5.7}
\Big|\na\Big\{z_\eps^\frac{p-1}{2}\cdot\big\{e^{-z_\eps}+\eps^\frac{1}{\alpha}\big\}^{-\frac{\alpha}{2}}\Big\}\Big|&=&\Big\{\frac{p-1}{2}z_\eps^\frac{p-3}{2} \big\{e^{-z_\eps}+\eps^\frac{1}{\alpha}\big\}^{-\frac{\alpha}{2}}+\frac{\alpha}{2}e^{-z_\eps}z_\eps^\frac{p-1}{2}\big\{e^{-z_\eps}+\eps^\frac{1}{\alpha}\big\}^{-\frac{\alpha}{2}-1}\Big\}\cdot |\na z_\eps|\nn\\
&\le& \frac{p(1+\alpha)}{2}z_\eps^\frac{p-1}{2}\big\{e^{-z_\eps}+\eps^\frac{1}{\alpha}\big\}^{-\frac{\alpha}{2}}\cdot|\na z_\eps|
\eea
for all $t>t_0$ and $\eps\in(0,\eps_{\star\star})$. Again, since $\uTeps+\eps^\frac{1}{\alpha}\le 1$, we have $1\le \big\{\uTeps+\eps^{\frac{1}{\alpha}}\big\}^{-\alpha}$ for all $t>t_0$ and $\eps\in(0,\eps_{\star\star})$. Thus \eqref{g5.5} yields with $p\ge 2$
\bea{g5.8}
D\io|\na z_\eps|^2z_\eps^{p-2}\frac{z_\eps+p-1}{\kappa_\eps(\uTeps)}\ge\frac{D}{c_5}\io z_\eps^{p-1}\big\{e^{-z_\eps}+\eps^\frac{1}{\alpha}\big\}^{-\alpha}\cdot|\na z_\eps|^2+\frac{Dp}{2c_5}\io z_\eps^{p-2}|\na z_\eps|^2 
\eea
for all $t>t_0$ and $\eps\in(0,\eps_{\star\star})$. From \eqref{4.16}, combined with \eqref{g5.7}, the definition of $z_\eps$ and our parameter $c_6$, we obtain
\bea{g5.9}
\frac{ac_5c_6}{p^2}&&\hspace{-10mm}\io z_\eps^{p-1}\cdot\big\{\uTeps+\eps^\frac{1}{\alpha}\big\}^{-\alpha}\nn\\
&=& \frac{a c_5c_6}{p^2}\Big\|z_\eps^\frac{p-1}{2}\cdot\big\{e^{-z_\eps}+\eps^\frac{1}{\alpha}\big\}^{-\frac{\alpha}{2}}\Big\|_{L^2(\Om)}^2\nn\\
&\le&\frac{ac_3c_5c_6}{p^2}\Big\|\na\Big\{z_\eps^\frac{p-1}{2}\cdot\big\{e^{-z_\eps}+\eps^\frac{1}{\alpha}\big\}^{-\frac{\alpha}{2}}\Big\}\Big\|_{L^2(\Om)}^2+\frac{a c_3c_5c_6}{p^2}\bigg|\int_{E_\eps(t)} z_\eps^\frac{p-1}{2}\cdot\big\{e^{-z_\eps}+\eps^\frac{1}{\alpha}\big\}^{-\frac{\alpha}{2}}\bigg|^2\nn\\
&\le& \frac{D}{c_5}\io z_\eps^{p-1}\big\{e^{-z_\eps}+\eps^\frac{1}{\alpha}\big\}^{-\alpha}\cdot|\na z_\eps|^2+\frac{ac_3c_5c_6}{p^2}\bigg|\int_{E_\eps(t)} z_\eps^\frac{p-1}{2}\cdot\big\{e^{-z_\eps}+\eps^\frac{1}{\alpha}\big\}^{-\frac{\alpha}{2}}\bigg|^2
\eea
for all $t>t_0$ and $\eps\in(0,\eps_{\star\star})$. Now set $r:=\frac{n+2}{n}>1$ and define $c_8:=\frac{D}{ac_4c_5^2c_7}$ and $c_9:=\frac{2}{n+2}r^{-\frac{n}{2}}c_8^{-\frac{n}{2}}$. Then the Gagliardo–Nirenberg inequality \eqref{GN-1} combined with Young's inequality yields
\bea{g5.10}
a c_5c_7 p^{2\alpha-2}\io z_\eps^{p-1}&\le& ac_5c_7 p^{2\alpha-2}\big\|z_\eps^\frac{p}{2}\big\|_{L^2(\Om)}^2\nn\\
&\le& ac_4c_5c_7p^{2\alpha-2}\big\|\na z_\eps^\frac{p}{2}\big\|_{L^2(\Om)}^\frac{2n}{n+2}\big\|z_\eps^\frac{p}{2}\big\|_{L^1(\Om)}^\frac{4}{n+2}+ac_4c_5c_7p^{2\alpha-2}\big\|z_\eps^\frac{p}{2}\big\|_{L^1(\Om)}^2\nn\\
&=&ac_4c_5c_7\cdot\Big\{\frac{1}{p}\big\|\na z_\eps^\frac{p}{2}\big\|_{L^2(\Om)}^2\Big\}^\frac{n}{n+2}\cdot\Big\{p^{n\alpha+2\alpha-\frac{n}{2}-2}\big\|z_\eps^\frac{p}{2}\big\|_{L^1(\Om)}^2\Big\}^\frac{2}{n+2}\nn\\
&&+ac_4c_5c_7p^{2\alpha-2}\big\|z_\eps^\frac{p}{2}\big\|_{L^1(\Om)}^2\nn\\
&\le&\frac{ac_4c_5c_7c_8}{p}\big\|\na z_\eps^\frac{p}{2}\big\|_{L^2(\Om)}^2+ac_4c_5c_7(c_9+1)p^{n\alpha+2\alpha-\frac{n}{2}-2}\big\|z_\eps^\frac{p}{2}\big\|_{L^1(\Om)}^2\nn\\
&=&\frac{Dp}{4c_5}\io z_\eps^{p-2}|\na z_\eps|^2+c_{10}p^{n\alpha+2\alpha-\frac{n}{2}-2}\big\|z_\eps^\frac{p}{2}\big\|_{L^1(\Om)}^2
\eea
for all $t>t_0$ and $\eps\in(0,\eps_{\star\star})$, with $c_{10}:=ac_4c_5c_7(c_9+1)$. Again the Gagliardo-Nirenberg inequality yields
\bea{nn}
\bigg\{\io z_\eps^p\bigg\}^\frac{n+2}{n}\le(2c_4)^\frac{n+2}{n}\big\|\na z_\eps^\frac{p}{2}\big\|_{L^2(\Om)}^2\cdot\bigg\{\io z_\eps^\frac{p}{2}\bigg\}^\frac{4}{n}+(2c_4)^\frac{n+2}{n}\bigg\{\io z_\eps^\frac{p}{2}\bigg\}^\frac{2n+4}{n}\ \text{for all }t>t_0\text{ and }\eps\in(0,\eps_{\star\star}),\ \ \ \nn
\eea
which immediately yields
\bea{g5.11}
-\frac{Dp}{4c_5}\io z_\eps^{p-2}|\na z_\eps|^2&=&-\frac{D}{pc_5}\big\|\na z_\eps^\frac{p}{2}\big\|_{L^2(\Om)}^2\nn\\
&\le&-\frac{D}{p(2c_4)^{\frac{n+2}{n}}c_5}\bigg\{\io z_\eps^p\bigg\}^\frac{n+2}{n}\cdot\bigg\{\io z_\eps^\frac{p}{2}\bigg\}^{-\frac{4}{n}}+\frac{D}{pc_5}\bigg\{\io z_\eps^\frac{p}{2}\bigg\}^2\nn\\
&\le& -\frac{c_{11}}{p}\bigg\{\io z_\eps^p\bigg\}^\frac{n+2}{n}\cdot\bigg\{\io z_\eps^\frac{p}{2}\bigg\}^{-\frac{4}{n}}+\frac{c_{12}}{p}\bigg\{\io z_\eps^\frac{p}{2}\bigg\}^2
\eea
for all $t>t_0$ and $\eps\in(0,\eps_{\star\star})$, with $c_{11}:=D (2c_4)^{-\frac{n+2}{n}}c_5^{-1}$ and $c_{12}:=\frac{D}{c_5}$. Upon combining \eqref{g5.8}, \eqref{g5.9}, \eqref{g5.10} and \eqref{g5.11} with \eqref{g5.6}, we deduce
\bea{nn}
\frac{d}{dt}\io z_\eps^p&\le& -c_{11}\bigg\{\io z_\eps^p\bigg\}^\frac{n+2}{n}\cdot\bigg\{\io z_\eps^\frac{p}{2}\bigg\}^{-\frac{4}{n}}+c_{12}\bigg\{\io z_\eps^\frac{p}{2}\bigg\}^2\nn\\
&&+\frac{ac_3c_5c_6}{p}\bigg|\int_{E_\eps(t)}z_\eps^\frac{p-1}{2}\cdot\big\{e^{-z_\eps}+\eps^\frac{1}{\alpha}\big\}^{-\frac{\alpha}{2}}\bigg|^2+c_{10}p^{n\alpha+2\alpha-\frac{n}{2}-1}\bigg\{\io z_\eps^\frac{p}{2}\bigg\}^2\nn
\eea
for all $t>t_0$ and $\eps\in(0,\eps_{\star\star})$. By \eqref{4.15}, $c_2$ is independent of $t$ and $\eps$, so that $1\le z_\eps\le -\ln c_2$ on $E_\eps(t)$ for all $t\in(t_0,t_0+\tau_0)$ and $\eps\in(0,\eps_\star)$, uniformly in $t$ and $\eps$. We choose $C_1:=c_{11}$, $C_2>0$ large enough and $C_3:=n\alpha+2\alpha-\frac{n}{2}-1$ such that \eqref{4.14} holds with $\eps_3:=\eps_{\star\star}$.\abs
Case 2: Assume that \eqref{A2} holds. Then by \eqref{kappa-eps-bed} there exists $c_{13}>0$ such that $\kappa_\eps(\xi)\ge c_{13}(\xi+\eps)$ for all $\xi\in(0,\eta)$ and $\eps\in(0,1)$. Due to Lemma \ref{lemThetaBounded} and \eqref{kappa-bed}, we can find $c_{14}>0$ fulfilling $\kappa_\eps(\uTeps)\le c_{14}$ in $\Om\times(t_0,\infty)$ for all $\eps\in(0,1)$. Now \eqref{g5.3} directly yields
\bea{nn}
\frac{d}{dt}\io z_\eps^p\le -\frac{Dp(p-1)}{c_{14}}\io |\na z_\eps|^2 z_\eps^{p-2}+\frac{ap}{c_{13}}\io z_\eps^{p-1} \qquad\text{for all }t>t_0\text{ and }\eps\in(0,1).\nn
\eea
By the same procedure as \eqref{g5.10}-\eqref{g5.11}, we obtain \eqref{4.14}.
\qed
We now perform a Moser iteration based on the above estimate.
\begin{lem}\label{lem4.8}
Let $\eta>0$ be such that $\kappa_\eps'(\xi)\ge0$ for all $\xi\in(0,\eta)$ and $\eps\in(0,1)$. Fix $\delta\in(0,\delta_1)$, where $\delta_1\in(0,\eta)$ and $\eps_2\in(0,1)$ are as in Lemma \ref{lem4.6}; this value of $\delta$ remains fixed for the remainder of this section. With $\tau_0>0$ as in Lemma \ref{lem3.4} and $\eps_3\in(0,1)$ as in Lemma \ref{Lem5.3} assume that \eqref{A1} holds for some $\alpha\ge 0$ or that \eqref{A2} holds, both on $(0,\eta)$. Let $\tau\in(0,\min\{1,\tau_0\})$, then there exists $C(\tau)>0$ such that for all $t_0\in[0,T]$ we have
\bea{nn}
\uTeps(x,t)\ge C(\tau)\qquad\text{for all }t\in(t_0+\frac{\tau}{2},t_0+\tau),\ x\in\Om\text{ and } \eps\in(0,\eps_4),\nn
\eea
with $\eps_4:=\min\{\eps_2,\eps_3\}$.
\end{lem}
\proof
Fix $k\ge1$ and define $p_k:=2^k$. From Lemma \ref{Lem5.3} we know the existence of $c_1>0$, $c_2>0$ and $c_3>0$ fulfilling
\bea{g5.12}
\frac{d}{dt}\io z_\eps^{p_k}\le -c_1\bigg\{\io z_\eps^{p_k}\bigg\}^{\frac{n+2}{n}}\cdot\bigg\{\io z_\eps^{\frac{p_k}{2}}\bigg\}^{-\frac{4}{n}}+c_2 p_k^{c_3}\bigg\{\io z_\eps^{\frac{p_k}{2}}\bigg\}^2
\eea
for all $t\in(t_0,t_0+\tau)$ and $\eps\in(0,\eps_4)$. We further define $\gamma_0:=1$ and for $k\ge 1$
\bea{g5.13}
\gamma_k:=2\gamma_{k-1}+\frac{n}{2}\quad\text{and}\quad c_{\eps k}:=\max\bigg\{\bigg(\frac{2\gamma_k}{c_1}\bigg)^\frac{n}{2}M_{\eps,k-1}^2,\ \frac{c_2p_k^{c_3} M_{\eps,k-1}^2}{\gamma_k},\ 1\bigg\},\quad\text{for }\eps\in(0,\eps_4),
\eea
where
\bea{g5.14}
M_{\eps,k}:=1+\sup_{t\in(0,\tau)}t^{\gamma_k}\io z_\eps(\cdot,t_0+t)^{p_k},\quad\text{for }k\ge 0\text{ and }\eps\in(0,\eps_4).
\eea
With $y_{\eps k}(t):=\io z_\eps(\cdot,t_0+t)^{p_k}$ for $t\in(0,\tau)$, $\eps\in(0,\eps_4)$, \eqref{g5.12} directly yields
\bea{g5.15}
y_{\eps k}'\le -c_1 y_{\eps k}^{\frac{n+2}{n}} t^{\frac{4\gamma_{k-1}}{n}} M_{\eps,k-1}^{-\frac{4}{n}}+c_2p_k^{c_3}M_{\eps,k-1}^{2}t^{-2\gamma_{k-1}}\qquad\text{for all }t\in(0,\tau)\text{ and }\eps\in(0,\eps_4).
\eea
Let us define
\bea{nn}
\Phi_{\eps k}(t):=c_{\eps k}t^{-\gamma_k},\qquad t\in(0,\tau)\text{ and }\eps\in(0,\eps_4).\nn
\eea
Since $\frac{4\gamma_{k-1}}{n}-\gamma_k\frac{n+2}{n}=\frac{4\gamma_{k-1}}{n}-\frac{2\gamma_k}{n}-\gamma_k\le -\gamma_k-1$ and $\tau\leq 1$, a direct computation shows that $\Phi_{\eps k}$ satisfies
$$t^{\frac{4\gamma_{k-1}}{n}}\Phi_{\eps k}^{\frac{n+2}{n}}\ge c_{\eps k}^{\frac{n+2}{n}}t^{-\gamma_k-1}\qquad\text{for all }t\in(0,\tau)\text{ and }\eps\in(0,\eps_4),$$
and combining this with \eqref{g5.13}, we immediately obtain
\bea{nn}
t^{\gamma_k+1}&&\hspace{-10mm}\cdot\big\{\Phi_{\eps k}'+c_1M_{\eps,k-1}^{-\frac{4}{n}} t^{\frac{4\gamma_{k-1}}{n}}\Phi_{\eps k}^{\frac{n+2}{n}}-c_2p_k^{c_3}M_{\eps,k-1}^2t^{-2\gamma_{k-1}}\big\}\nn\\
&\ge& -\gamma_k c_{\eps k}+c_1 c_{\eps k}^{\frac{n+2}{n}} M_{\eps,k-1}^{-\frac{4}{n}}-c_2p_k^{c_3} M_{\eps,k-1}^2 t^{\gamma_k+1-2\gamma_{k-1}}\nn\\
&=& c_{\eps k}(c_1c_{\eps k}^{\frac{2}{n}}M_{\eps,k-1}^{-\frac{4}{n}}-\gamma_k)-c_2p_k^{c_3} M_{\eps,k-1}^2 t^{\gamma_k+1-2\gamma_{k-1}}\nn\\
&\ge& c_{\eps k}(2\gamma_k-\gamma_k)-c_2p_k^{c_3}M_{\eps,k-1}^2\ge 0\qquad\text{for all }t\in(0,\tau)\text{ and }\eps\in(0,\eps_4),\nn
\eea
that is, $\Phi_{\eps k}$ is a supersolution to the differential inequality satisfied by $y_{\eps k}$ in \eqref{g5.15} for all $\eps\in(0,\eps_4)$. Since the right-hand side of \eqref{g5.15} is strictly decreasing in the solution variable $y_{\eps k}$, and since $$\lim_{t\searrow 0}\Phi_{\eps k}(t)=\infty,\qquad\text{for all }\eps\in(0,\eps_4)$$
a comparison argument yields
\bea{g5.20}
y_{\eps k}(t)\le\Phi_{\eps k}(t)=c_{\eps k} t^{-\gamma_k}\qquad\text{for all }t\in(0,\tau)\text{ and }\eps\in(0,\eps_4).
\eea
Due to the definition of $\gamma_k$ in \eqref{g5.13}, there exists a constant $c_4>0$ such that
\bea{g_k}
\gamma_k\le c_4 2^k\qquad\text{for all }k\ge 1.
\eea
This allows us to choose constants $c_5>0$ and $c_6>0$ satisfying
$$2\bigg(\frac{2 \gamma_k}{c_1}\bigg)^\frac{n}{2}\le 2\bigg(\frac{c_42^{k+1}}{c_1}\bigg)^\frac{n}{2}\le c_5^k\quad\text{and}\quad 2\frac{c_2p_k^{c_3}}{\gamma_k}\le 2^{kc_3+1}c_2\le c_6^k\qquad\text{for all }k\ge1,$$
so from \eqref{g5.13}, \eqref{g5.14} and \eqref{g5.20} we finally obtain
\bea{nn}
M_{\eps,k}\le 2c_{\eps k}\le c_7^k M_{\eps,k-1}^{2}\qquad\text{for all }k\ge1\text{ and }\eps\in(0,\eps_4),\nn
\eea
with $c_7:=\max\{c_5,c_6,1\}$. Thus we immediately obtain
$$M_{\eps,k}\le c_7^{\sum_{j=1}^k j2^{k-j}}M_{\eps,0}^{2^k}=c_7^{2^{k+1}-k-2}M_{\eps,0}^{2^k}\le c_7^{2^{k+1}}M_{\eps,0}^{2^k}\qquad\text{for all }k\ge1\text{ and }\eps\in(0,\eps_4)$$
and finally
$$M_{\eps,k}^\frac{1}{2^k}\le c_7^2M_{\eps,0}\qquad\text{for all }k\ge 1\text{ and }\eps\in(0,\eps_4).$$
From this, 
$$\sup_{t\in(0,\tau)} t^\frac{\gamma_k}{2^k} \bigg\{\int_\Omega z_\eps(\cdot,t_0+t)^{2^k}\bigg\}^\frac{1}{2^k} \le c_7^2 M_{\eps,0}\qquad\text{for all }k\ge 1\text{ and }\eps\in(0,\eps_4).$$
From \eqref{g_k} we get $t^\frac{\gamma_k}{2^k}\ge t^{c_4}$ for all $t\in(0,\tau)\subset(0,1)$, thus taking $k \to \infty$ yields
$$\|z_\eps(\cdot,t_0+t)\|_{L^\infty(\Om)} \le t^{-c_4} c_7^2 M_{\eps,0} \qquad \text{for all } t \in (0,\tau) \text{ and } \varepsilon \in (0,\varepsilon_3).$$
According to Lemma \ref{lem4.6} and \eqref{g5.14} we can fix $t\in(t_0+\frac{\tau}{2},t_0+\tau)$ and find $c_8=c_8(\tau)>0$ such that
$$\|z_\eps(\cdot,t)\|_{L^\infty(\Om)}\le c_8\qquad\text{for all }\eps\in(0,\eps_4),$$
which entails that $\uTeps(x,t)\ge e^{-c_8}$ for all $x\in\Om$ and $\eps\in(0,\eps_4)$.
\qed
Iterating the above argument, we finally arrive at our desired result:\abs
{\sc Proof of Theorem \ref{theo2}.} \quad Lemma \ref{lem4.8} provides the bound on the initial interval $(\frac{\tau}{2}, \tau)$. To extend it to $(\tau, T)$ for arbitrary $T > \tau$, we iterate forward in time: by Lemma \ref{W1}, every time level $t_0$ possesses a set of positive measure where $\Theta_\eps$ is bounded away from zero. Thus, applying Lemma \ref{lem4.8} repeatedly with starting times $t_0^{(k)} = \frac{k\tau}{4}$, the resulting intervals cover $(\tau, T)$. By pointwise convergence \eqref{2.23}, the assertion follows.
\qed
\bigskip
{\bf Acknowledgment.} \quad
The author acknowledges support of the Deutsche Forschungsgemeinschaft (Project No. 444955436).
He moreover declares that he has no conflict of interest.\abs
{\bf Data availability statement.} \quad
Data sharing is not applicable to this article as no datasets were
generated or analyzed during the current study.

\end{document}